 \documentclass[final,1p,times]{elsarticle}
\usepackage{amssymb}
 \usepackage{amsthm}
\usepackage{amsmath,amssymb,amsopn,amsfonts,mathrsfs,amsbsy,amscd}

\newcommand{\prs}{\langle\;,\;\rangle}
\newcommand{\too}{\longrightarrow}

\newcommand{\esp}{\quad\mbox{and}\quad}

\newcommand{\G}{{\mathfrak{g}}}

\newcommand{\ad}{{\mathrm{ad}}}
\newcommand{\tr}{{\mathrm{tr}}}

\newcommand{\wi}{\widetilde}
\newcommand{\al}{\alpha}
\newcommand{\be}{\beta}

\newcommand{\e}{\epsilon}

\newcommand{\la}{\lambda}

\font\bb=msbm10

\def\R{\hbox{\bb R}}

\def\N{\hbox{\bb N}}

\def\C{\hbox{\bb C}}

\newtheorem{theo}{Theorem}[section]
\newtheorem{pr}{Proposition}[section]

\newtheorem{co}{Corollary}[section]

\begin{document}

\begin{frontmatter}

%% Title, authors and addresses

%% use the tnoteref command within \title for footnotes;
%% use the tnotetext command for the associated footnote;
%% use the fnref command within \author or \address for footnotes;
%% use the fntext command for the associated footnote;
%% use the corref command within \author for corresponding author footnotes;
%% use the cortext command for the associated footnote;
%% use the ead command for the email address,
%% and the form \ead[url] for the home page:
%%
 %\title{Left invariant para-K\"ahler and hyper-para-K\"ahler structures on Lie groups\tnoteref{label1}}
%% \tnotetext[label1]{}
 %\author{\corref{cor1}\fnref{label2}}
 
%% \ead{email address}
%% \ead[url]{home page}
%% \fntext[label2]{}
%% \cortext[cor1]{}
%% \address{Address\fnref{label3}}
 %\fntext[label3]{This research was conducted within the framework of Action concert\'ee CNRST-CNRS Project SPM04/13.}

\title{ Nonunimodular Lorentzian flat Lie algebras}

%% use optional labels to link authors explicitly to addresses:
 \author[Mohamed Boucetta]{Mohamed Boucetta}
 \address[Mohamed Boucetta]{Universit\'e Cadi-Ayyad\\
  Facult\'e des sciences et techniques\\
  BP 549 Marrakech Maroc
 }
 \author[Hicham Lebzioui]{Hicham Lebzioui\corref{mycorrespondingauthor}}
%\cortext[mycorrespondingauthor]{Corresponding author}
%\ead{hlebzioui@gmail.com}
\address[Hicham Lebzioui]{Universit\'e Moulay Smail\\
 Facult\'e des sciences\\BP 11201 Zitoune, Mekn\`{e}s - Maroc
 }
%\author{}

%\address{}
\begin{abstract}A {\it Lorentzian flat Lie group} is a Lie group $G$ with a flat left invariant 
  metric $\mu$ with signature $(1,n-1)=(-,+,\ldots,+)$. The Lie algebra $\G=T_eG$ of $G$ endowed with 
$\prs=\mu(e)$ is called
{\it flat Lorentzian   Lie algebra}. It is known that the metric of a flat Lorentzian Lie group is geodesically complete if and only if its Lie algebra is unimodular. In this paper, we characterise nonunimodular Lorentzian flat Lie algebras as double extensions (in the sense of Aubert-Medina \cite{Aub-Med}) of Riemannian flat Lie algebras. As application of this result, we give all nonunimodular Lorentzian flat Lie algebras up to dimension 4. 

\end{abstract}
\begin{keyword}Lorentzian flat Lie algebras \sep Nonunimodular Lie algebras \sep Representations of solvable Lie algebras\sep Double extension %% keywords here, in the form: keyword \sep keyword
\MSC 53C50 \sep \MSC 53D20 \sep \MSC 17B62\sep \MSC 16T25

%% MSC codes here, in the form: \MSC code \sep code
%% or \MSC[2008] code \sep code (2000 is the default)

\end{keyword}

\end{frontmatter}

\section{Introduction and main results}

A {\it pseudo-Riemannian  Lie group} is a Lie group $G$ with a left invariant 
pseudo-Riemannian  metric $\mu$. The Lie algebra $\G=T_eG$ of $G$ endowed with 
$\prs=\mu(e)$ is called
{\it pseudo-Riemannian   Lie algebra}. The Levi-Civita
connection of $(G,\mu)$ defines a product $(u,v)\mapsto u.v$ on $\G$ called \emph{Levi-Civita
product} given by  Koszul's formula
$$2\langle
u.v,w\rangle=\langle[u,v],w\rangle+\langle[w,u],v\rangle+\langle[w,v],u\rangle.
$$
For any $u\in\G$, we denote by $\mathrm{L}_u:\G\too\G$ and $\mathrm{R}_u:\G\too\G$,
respectively, the left multiplication and the right multiplication by $u$ given by
$\mathrm{L}_uv=u.v$ and $\mathrm{R}_uv=v.u.$ For any $u\in\G$, $\mathrm{L}_u$ is
skew-symmetric with
respect to $\prs$ and
$\ad_u=\mathrm{L}_u-\mathrm{R}_u,$ where
$\ad_u:\G\too\G$ is given by $\ad_uv=[u,v]$. 
The curvature of $\mu$ at $e$ is given by
$$\mathrm{K}(u,v)=\mathrm{L}_{[u,v]}-[\mathrm{L}_u,\mathrm{L}_v].$$
If $\mathrm{K}$ vanishes then   $(G,\mu)$ is called  \emph{pseudo-Riemannian flat Lie group} and $(\G,\prs)$ is called  \emph{pseudo-Riemannian flat Lie
algebra}.
If $\mu$ is geodesically complete then $(\G,\prs)$ is called \emph{ complete}.\\
 A pseudo-Riemannian flat Lie
algebra is
complete if and only if it is unimodular (see \cite{Aub-Med}). A Riemannian (resp. Lorentzian)  Lie group
is a pseudo-Riemannian  Lie group for which the metric is definite positive (resp. of
signature $(-,+\ldots+)$). In \cite{Milnor}, Milnor showed that a Riemannian Lie group is  flat
  if and only
if its Lie algebra is a semi-direct product of an abelian algebra $\mathfrak{b}$ with an
abelian ideal $\mathfrak{u}$ and, for any $u\in\mathfrak{b}$, $\ad_u$ is skew-symmetric. The determination of Lorentzian flat Lie groups is an open problem. A Lorentzian flat Lie algebra must be solvable (see
\cite{della}). In \cite{Aub-Med}, Aubert and Medina showed that nilpotent Lorentzian flat
Lie algebras are obtained by the double extension process from Riemannian abelian Lie
algebras. In \cite{gue}, Guediri studied Lie groups which may act isometrically and simply transitively on Minkowski space and get a precise description of nilpotent Lorentzian flat Lie groups.
 In \cite{ABL}, the authors  showed that a Lorentzian flat Lie
algebras with degenerate center can be obtained by the double extension process from Riemannian flat Lie
algebras.  In the first part of this paper, we show that  any nonunimodular Lorentzian flat Lie algebra is obtained by the double extension process from a Riemannian flat Lie algebra. In the second part, as application of this result, we determine all nonunimodular Lorentzian flat Lie algebras up to dimension 4.
Let us state our main result in a more precise way. To do so,
 we need to recall some basic material.
 \begin{itemize} 
 \item The double extension process  was described in \cite{Aub-Med}. In particular, Propositions
3.1-3.2 of 
\cite{Aub-Med} are essential in this process.
  Let
$(B,[\;,\;]_0,\prs_0)$ be a pseudo-Riemannian flat Lie algebra,
$\xi,D:B\too B$ two
endomorphisms of $B$, $b_0\in B$ and $\mu\in\R$ such that:
\begin{enumerate}
 \item \label{enu1}$\xi$ is a 1-cocycle of  $(B,[\;,\;]_0)$ with respect to the
representation $\mathrm{L}:B\too\mathrm{End}(B)$ defined by the left multiplication
associated to the Levi-Civita product, i.e., for any $a,b\in B$,
\begin{equation}\label{eq3}
 \xi([a,b])=\mathrm{L}_a\xi(b)-\mathrm{L}_b\xi(a),
\end{equation}

\item $D-\xi$ is skew-symmetric with respect to $\prs_0$, 
\begin{equation} \label{eq5}
[D,\xi]=\xi^2-\mu\xi-\mathrm{R}_{b_0},
\end{equation}and for any $a,b\in B$
\begin{equation} \label{eq6}
a.\xi(b)-\xi(a.b)=D(a).b+a.D(b)-D(a.b).
\end{equation}

\end{enumerate}
 We call  $(\xi,D,\mu,b_0)$ satisfying the two conditions above
\emph{admissible}.\\ Given $(\xi,D,\mu,b_0)$ admissible, 
we endow the vector space $\G=\R e\oplus B\oplus\R \bar{e}$  with the inner product
$\prs$ which extends $\prs_0$, for which $\mathrm{span}\{e,\bar{e}\}$ and $B$ are
orthogonal,
 $\langle e,e\rangle=\langle \bar{e},\bar{e}\rangle=0$ and $\langle e,\bar{e}\rangle=1$.
We define
also on
$\G$ the bracket 
\begin{equation}\label{bracket}[\bar{e},e]=\mu e,\; [\bar{e},a]=D(a)-\langle
b_0,a\rangle_0e\esp[a,b]=[a,b]_0+\langle(\xi-\xi^*)(a),b\rangle_0e,\end{equation}where
$a,b\in B$
and
$\xi^*$ is the adjoint of $\xi$ with respect to $\prs_0$. Then $(\G,[\;,\;],\prs)$ is a
pseudo-Riemannian flat Lie algebra  called \emph{double extension}
of $(B,[\;,\;]_0,\prs_0)$ according to $(\xi,D,\mu,b_0)$. \item
It was proven in \cite{ABL} that if $(B,[\;,\;],\prs_0)$ is a Riemannian flat Lie algebra then $B$ splits orthogonally
 \[ B=S(B)\oplus Z(B)\oplus[B,B], \]where $Z(B)$ is the center of $B$,
 $$ S(B)\oplus Z(B)=[B,B]^\perp=\{b\in B, \mathrm{R}_b=0\}=
 \{b\in B, \ad_b+\ad_b^*=0\}
 ,$$ and, for any $b\in Z(B)\oplus [B,B]$, $\mathrm{L}_b=0$.
 Moreover, $B.B=[B,B]$ and $\dim[B,B]$ is even.
\item 
 The \emph{modular vector} of a pseudo-Riemannian Lie algebra $(\G,\prs)$ is
the vector $\mathbf{h}\in\G$
given by
\begin{equation}\label{mean}\langle \mathbf{h},u\rangle=\mathrm{tr(\ad_u)}
=-\tr(\mathrm{R}_u),\;\forall
u\in\G.\end{equation}The
Lie algebra $\G$ is unimodular if and only if $\mathbf{h}=0$. Denote by $H=\mathrm{span}\{\mathbf{h} \}$ and $H^\perp$ its orthogonal with respect to $\prs$.\end{itemize}

We can now state our main result.
\begin{theo}\label{main} Let $(\G,\prs)$ be a nonunimodular Lorentzian flat Lie algebra. Then:
\begin{enumerate}\item [$(i)$] The left multiplication by $\mathbf{h}$ vanishes, i.e., $\mathrm{L}_{\mathbf{h}}=0$ and  both $H$ and $H^\perp$ are two-sided ideals with respect to the Levi-Civita product.  
\item [$(ii)$] $(\G,\prs)$ is obtained by the double extension process from a Riemannian flat Lie algebra $(B,[\;,\;]_0,\prs_0)$ according to $(\xi,D,\mu,b_0)$ with $\tr(D)\not=-\mu$.  
\end{enumerate}

\end{theo}

 The proof of Theorem \ref{main}  is given  in Section \ref{section4}. It is based on a property of the modular vector given in Proposition \ref{pr1}, and on the fact that a  Lorentzian representation of a solvable Lie algebra can be reduced in an useful way by virtue of Lie's Theorem (see \cite{knap} Theorem 1.25 pp. 42).  Section \ref{section3} is devoted to the study of Lorentzian representations of  solvable Lie algebras. In Section \ref{section5} we give all nonunimodular Lorentzian flat Lie algebras up to dimension 4.

\section{Lorentzian representations of solvable Lie algebras }\label{section3}

In this section, by using Lie's Theorem (see \cite{knap} Theorem 1.25 pp. 42), we derive some interesting results on
Euclidean and Lorentzian representations of solvable Lie algebras. Through this section, $\G$ is a real solvable Lie algebra. We fix an ordering on $\G^*$ and,
for any $\la\in\G^*$, we denote by $d\la$ the element of $\wedge^2\G^*$ given by $d\la(u,v)=-\la([u,v])$.

A \emph{pseudo-Euclidean  vector space } is  a real vector space  of finite dimension $n$
endowed with  a
nondegenerate symmetric inner product  of signature $(q,n-q)=(-,\ldots,-,+,\ldots,+)$.  When
the
signature is $(0,n)$
(resp. $(1,n-1)$) the space is called \emph{Euclidean} (resp. \emph{Lorentzian}). 
Let $(V,\prs)$ be a pseudo-Euclidean vector space whose signature is $(q,n-q)$, we denote by $\mathrm{so}(V)$ the Lie algebra of skew-symmetric endomorphisms of $(V,\prs)$. 
Let  $\rho:\G\too\mathrm{so}(V)$ be a representation of $\G$. For any $\la\in\G^*$, put $$V_\la=\{x\in V,\;\rho(u)x=\la(u)x\;\mbox{for all}\;u\in\G \}.$$The representation $\rho$  is called \emph{indecomposable} if $V$ does not contain any nondegenerate invariant vector subspace. 
\begin{pr}Let $\rho:\G\too\mathrm{so}(V)$ be  an indecomposable representation  on an Euclidean vector space. Then either $\dim V=1$ and $V=V_0$ or $\dim V=2$,  there exists $\la>0$ such that $d\la=0$ and, for any $u\in\G$, $\rho(u)^2=-\la(u)^2\mathrm{Id}_V$. Moreover, in the last case there exists an orthonormal basis $(e,f)$ of $V$ such that, for any $u\in\G$,
\[ \rho(u)(e)=\la(u)f\esp \rho(u)(f)=-\la(u)e, \]and any other orthonormal basis of $V$ satisfying these relations has the form $(\cos(\psi) e-\sin(\psi) f,\sin(\psi) e+\cos(\psi)f)$, $\psi\in\R$.
\label{pr2}

\end{pr}

{\bf Proof.}
We consider the complexification $V^{\C}$ of $V$. Then $\rho$ extends to a representation $\rho^{\C}:\G\too \mathrm{End}_{\C}(V^{\C})$ by putting
$$\rho^{\C}(u)(a+\imath b)=\rho(u)(a)+\imath\rho(u)(b).$$
Since $\G$ is solvable then by virtue of Lie's Theorem there exists $\la_1+\imath \la_2:\G\too\C$ and $x+\imath  y\not=0$ such that for any $u\in\G$,
$$\rho^{\C}(u)(x+\imath y)=(\la_1(u)+\imath \la_2(u))(x+\imath y).$$
This is equivalent to
\begin{equation}\label{eq1} \rho(u)(x)=\la_1(u)x-\la_2(u)y\esp
\rho(u)(y)=\la_2(u)x+\la_1(u)y.  \end{equation}
From
\[ \langle\rho(u)x,x\rangle=\langle\rho(u)y,y\rangle=0
\esp \langle\rho(u)x,y\rangle=-\langle\rho(u)y,x\rangle \]we get
\begin{equation}\label{eq21} \left(\begin{array}{ccc}\la_1(u)&-\la_2(u)&0\\
0&\la_2(u)&\la_1(u)\\ \la_2(u)&2\la_1(u)&-\la_2(u)\end{array} \right)
\left(\begin{array}{c}\langle x,x\rangle\\\langle y,x\rangle\\\langle y,y\rangle \end{array}\right)=0. \end{equation}We distinguish two cases:
\begin{enumerate}\item[$(i)$] The vectors $x,y$ are linearly dependent, for instance $y=ax$ and $x\not=0$, then
\[ \rho(u)(x)=(\la_1(u)-a\la_2(u))x\esp
a\rho(u)(x)=(\la_2(u)+a\la_1(u))x.  \]
Then
\[ \la_2(u)+a\la_1(u)=a\la_1(u)-a^2\la_2(u), \] hence $\la_2=\la_1=0$ so $x\in V_{0}\not=\{0\}$. The representation being indecomposable implies that $\dim V=1$ and $V=V_0$.
\item[$(ii)$] The couple $(x,y)$ are linearly independent.
Since $\langle x,x\rangle\not=0$ then, from \eqref{eq21}, we get 
\[ \left|\begin{array}{ccc}\la_1(u)&-\la_2(u)&0\\
0&\la_2(u)&\la_1(u)\\ \la_2(u)&2\la_1(u)&-\la_2(u)\end{array} \right|=-2\la_1(u)
\left(\la_1(u)^2+\la_2(u)^2  \right)=0.\]
If $\la_1=\la_2=0$ then $\mathrm{span}\{x,y\}\subset V_0$ which impossible. If $\la_1=0$ and $\la_2\not=0$ then $\langle x,y\rangle=0$  and the restriction of $\prs$ to 
$\mathrm{span}\{x,y\}$ is  nondegenerate  and hence $V=\mathrm{span}\{x,y\}$. By using the fact that $\rho([u,v])=[\rho(u),\rho(v)]$, one can deduce that $d\la_2=0$ and the  proposition follows.\hfill$\square$
\end{enumerate}

\begin{co}\label{co1}Let $\rho:\G\too\mathrm{so}(V)$ be  a  representation  on an Euclidean vector space. Then $V$ splits orthogonally
$$V=\bigoplus_{i=1}^qE_i\oplus V_0,$$where $E_i$ is an invariant indecomposable 2-dimensional
vector space for $i=1,\ldots,q$. In particular, a solvable subalgebra of $\mathrm{so}(V)$ must be abelian.

\end{co}

\begin{pr}\label{pr3}
Let $\rho:\G\too \mathrm{so}(V)$ be an indecomposable Lorentzian representation. Then one of the following cases occurs:
\begin{enumerate}\item $\dim V=1$ and $V=V_0$.
\item $\dim V=2$,  there exists $\la>0$ such that $d\la=0$ and   a basis $(e,\bar{e})$ of $V$ such that $\langle e,e\rangle=\langle \bar{e},\bar{e}\rangle=0$, $\langle e,\bar{e}\rangle=1$
and, for any $u\in\G$,
\[ \rho(u)e=\la(u){e}\esp\rho(u)\bar{e}=-\la(u)\bar{e}. \]
\item $\dim V\geq3$, there exists $\la\in\G^*$ such that $d\la=0$ and $V_\la$ is a totally isotropic one dimensional vector space. Moreover, for any $\mu\not=\la$, $V_\mu=\left\{0\right\}$.

\end{enumerate}

\end{pr}

{\bf Proof.} As in the proof of Proposition \ref{pr2}, by virtue of Lie's Theorem, there exists $\la_1,\la_2\in\G$ and $x,y\in V$, $(x,y)\not=(0,0)$,  satisfying \eqref{eq1}-\eqref{eq21}. We distinguish two cases:\begin{enumerate}
\item[$(a)$] The vectors $x,y$ are linearly dependent say $y=ax$ with $x\not=0$. From \eqref{eq1}-\eqref{eq21}, we get 
 $\la_2=0$ and, for any $u\in\G$, $\la_1(u)\langle x,x\rangle=0$. 
If $\langle x,x\rangle\not=0$ then $\dim V=1$,  $V=V_0$ and we are in the first case. \\ Suppose now that $\langle x,x\rangle=0$. If   $V_{\la_1}$ is non totally isotropic then it contains a non isotropic vector $z$ and hence $V=\mathrm{span}\{z\}$ which is impossible since $x\in V$. So
   $V_{\la_1}$ must be totally isotropic and hence  $ V_{\la_1}=\mathrm{span}\{x\}$.   We have then two situations. The first one is that there exists $\mu\not=\la_1$ such that $V_\mu=\mathrm{span}\{z\}$ is a totally isotropic one dimensional vector space.  From the relation $\langle\rho(u)x,z\rangle=-\langle\rho(u)z,x\rangle$ and $\langle x,z\rangle\not=0$, we deduce that
$\mu=-\la_1$. Then $\la_1\not=0$ and hence $V=V_{\la_1}\oplus V_\mu$ and we are in the second case.  The second situation is that, for any $\mu\not=\la_1$, $V_\mu=\{0\}$. In this case $\dim V\geq3$ and we are in the third case. Indeed,
if $\dim V=2$, choose an isotropic vector $\bar{x}$ such that $\langle x,\bar{x}\rangle=1$. It is easy to check that $\bar{x}\in V_{-\la_1}$ which is impossible.

\item[$(b)$] The vectors $x,y$ are linearly independent. Since $\mathrm{span}\{x,y\}$ cannot be totally isotropic, we can deduce from \eqref{eq21} that $\la_1=0$, $\la_2\not=0$,  
$\langle x,y\rangle=0$ and $\langle x,x\rangle=\langle y,y\rangle\not=0$. So 
$\mathrm{span}\{x,y\}$ is Euclidean nondegenerate invariant which is impossible.
\hfill$\square$
\end{enumerate}

Let us study the third case in Proposition \ref{pr3} more deeply. Let $\rho:\G\too \mathrm{so}(V)$ be an indecomposable Lorentzian representation with $\dim V\geq3$. Then
 there exists $\la\in\G^*$ such that $d\la=0$, $V_\la$ is a one dimensional totally isotropic subspace and, for any $\mu\not=\la$, $V_\mu=\{0\}$. The quotient
 $\wi V=V_\la^\perp/V_\la$ is an Euclidean vector space and $\rho$ induces a representation  $\wi\rho:\G\too\mathrm{so}(\wi V)$. So, according to Corollary \ref{co1},
\[ \wi V=\bigoplus_{i=1}^q\wi E_i\oplus \wi V_0. \]
Denote by $\pi: V_\la^\perp\too\wi V$ the natural projection and choose a generator of $V_\la$. De note by $E_i=\pi^{-1}(\wi E_i)$, for $i=1,\ldots,q$, and $E_0=\pi^{-1}(\wi V_0)$.\\
For any $x\in E_0$ there exists $a_x\in\G^*$ such that, for any $u\in\G$,
\begin{equation}\label{x} \rho(u)(x)=a_x(u)e. \end{equation}This defines a linear map $a:E_0\too\G^*$. Its kernel $\ker a$ is an invariant vector subspace so, since $V$ is indecomposable, 
\begin{equation} \label{ker}\ker a=\left\{\begin{array}{ccc}\{0\}&\mbox{if}&\la\not=0,\\
\R e&\mbox{if}&\la=0.\end{array}    \right. \end{equation}
For any $x\in E_0$ and any $u,v\in\G$,
\begin{eqnarray*} \rho([u,v])(x)&=&a_x([u,v])e\\&=&(\la(u)a_x(v)-\la(v)a_x(u))e. \end{eqnarray*}
 Thus, for any $x\in E_0$, 
 \begin{equation}\label{ax}
  da_x=a_x\wedge\la.
 \end{equation}
Fix $i=1,\ldots,q$ and, by using Proposition \ref{pr1}, choose a basis
  $(e,e_i,f_i)$ is a basis of $E_i$ such that, for any $u\in\G$,
  \begin{equation}\label{efi} \rho(u)(e)=\la(u)e,\;\rho(u)e_i=b_i(u)e+\la_i(u)f_i\esp
   \rho(u)f_i=c_i(u)e-\la_i(u)e_i.\end{equation}
    We have, for any $u,v\in\G$,
      \begin{eqnarray*}
      \rho([u,v])e_i&=&b_i([u,v])e+\la_i([u,v])f_i\\
      &=&\rho(u)(b_i(v)e+\la_i(v)f_i)-\rho(v)(b_i(u)e+\la_i(u)f_i)\\&=&(\la(u)b_i(v)-\la(v)b_i(u)+\la_i(v)c_i(u)-\la_i(u)
        c_i(v))e,
      \end{eqnarray*}so
      \[ b_i([u,v])=\la(u)b_i(v)-\la(v)b_i(u)+\la_i(v)c_i(u)-\la_i(u)
          c_i(v). \]
   In the same way, by computing $\rho([u,v])f_i$, we get
      \[ c_i([u,v])=\la(u)c_i(v)-\la(v)c_i(u)-\la_i(v)b_i(u)+\la_i(u)
          b_i(v). \]
    Thus
    \begin{equation}\label{bci} db_i= b_i\wedge\la+\la_i\wedge c_i\esp d c_i=c_i\wedge\la+b_i\wedge \la_i.  \end{equation}
    If $(e,g_i,h_i)$ is another basis satisfying \eqref{efi} then
   \[ g_i=\al e+\cos(\psi) e_i+\sin(\psi) f_i\esp h_i=\be e
   -\sin(\psi) e_i+\cos(\psi) f_i. \]
   So
   \begin{eqnarray*} \rho(u)(g_i)&=&(\al\la(u)-\be\la_i(u)+\cos(\psi) b_i(u)+\sin(\psi) c_i(u))e\\&&+\la_i(u)(\cos(\psi) f_i-\sin(\psi) e_i+\be e)\\
   &=&B_i(u)e+\la_i(u)h_i,\\ 
   \rho(u)(h_i)&=&(\be\la(u)+\al\la_i(u)-\sin(\psi) b_i(u)+\cos(\psi) c_i(u))e\\&&-\la_i(u)(\sin(\psi) f_i+\cos(\psi) e_i+\al e)\\
     &=&C_i(u)e-\la_i(u)g_i.
   \end{eqnarray*}So , for any $u\in\G$,
   \begin{equation}\label{changebc}
   \left(\begin{array}{c}B_i(u)\\C_i(u) \end{array}\right)= \left(\begin{array}{cc}\al&-\be\\\be&\al\end{array} \right)
   \left(\begin{array}{c}\la(u)\\\la_i(u) \end{array}\right)+\left(\begin{array}{cc} \cos(\psi)&\sin(\psi)\\
   -\sin(\psi)&\cos(\psi)\end{array}\right)
   \left(\begin{array}{c}b_i(u)\\c_i(u) \end{array}\right). \end{equation}
  The case when $\dim\G=1$ is useful and leads to the following proposition which is a part of the folklore.
  \begin{pr}Let $F$ be a skew-symmetric endomorphism of a Lorentzian vector space $L$. Then $L$ splits orthogonally $L=V\oplus E$, where $V$ is  nondegenerate Lorentzian invariant  and $E$ is nondegenerate Euclidean invariant and one of the following cases occurs:
    \begin{enumerate}\item[$(i)$] $\dim V=1$ and $V\subset\ker F$.
    \item[$(ii)$] $\dim V=2$ and there exists a basis $(e,\bar{e})$ of $V$ and $\al>0$ such that, $\langle e,e\rangle=\langle \bar{e},\bar{e}\rangle=0$, $\langle \bar{e},{e}\rangle=1$,  $F(e)=\al e$ and $F(\bar{e})=-\al\bar{e}$.
    \item[$(iii)$] $\dim V=3$ and there exists a basis $(e,\bar{e},f)$ of $V$ and $\al\not=0$ such that, $$\langle e,e\rangle=\langle \bar{e},\bar{e}\rangle=0,\;\langle \bar{e},{e}\rangle=\langle f,f\rangle=1,\;\langle e,f\rangle=\langle \bar{e},{f}\rangle=0,$$ $F(e)=0$, $F(f)=\alpha e$ and $F(\bar{e})=-\alpha f.$   \end{enumerate}

  \end{pr}
  
  {\bf Proof.} Consider the one dimensional Lie algebra $\G$ spanned by $F$. We have obviously $L=V\oplus E$ where $V$ is nondegenerate Lorentzian invariant indecomposable and $E$ is Euclidean nondegenerate invariant. By applying Proposition \ref{pr3} to the representation of $\G$ on $V$, we get obviously $(i)$ and $(ii)$. Suppose now that $\dim V\geq3$. Then
   there exists $\la\in\G^*$ such that $d\la=0$, $V_\la$ is a one dimensional totally isotropic subspace and, for any $\mu\not=\la$, $V_\mu=\{0\}$. We adopt the notations used in the study above. For any $i=1,\ldots,q$,  since $\la_i(F)\not=0$, by using \eqref{changebc}, we can choose a basis $(e,e_i,f_i)$ of $E_i$ such that $b_i(e)=c_i(e)=0$ and hence $\mathrm{span}\{e_i,f_i\}$ is nondegenerate invariant. Thus $V_\la^\perp=E_0$. Now consider the endomorphism $a:E_0\too\G^*$. If $\la\not=0$ then, according to \eqref{ker}, $\dim E_0=1$ and hence $\dim V=2$ which is impossible. So $\la=0$ and hence $\dim E_0=2$ which establishes $(iii)$.\hfill$\square$

\section{Proof of Theorem \ref{main}}\label{section4}
Before to prove Theorem \ref{main}, we establish an important property of the modular vector of a {pseudo-Riemannian flat Lie algebra} which will be crucial in the proof.\\
Let $(\G,\prs)$ be a {pseudo-Riemannian flat Lie algebra}. The vanishing of the curvature is equivalent to the fact that $\G$ endowed with the Levi-Civita product
 is a left symmetric
 algebra, i.e., for any $u,v,w\in\G$,
 $$\mathrm{ass}(u,v,w)=\mathrm{ass}(v,u,w),$$where
 $\mathrm{ass}(u,v,w)=(u.v).w-u.(v.w).$ This relation is equivalent to
 \begin{equation}\label{left}\mathrm{R}_{u.v}-\mathrm{R}_v\circ\mathrm{R}_u=[\mathrm{L}_u,
 \mathrm{R}_v
 ],\end{equation}for any $u,v\in\G$. On the other hand,  one can see easily that the
 orthogonal of
 the derived ideal of $\G$ is given by
 \begin{equation}
 \label{eq7}[\G,\G]^\perp=\{u\in\G,\mathrm{R}_u=\mathrm{R}_u^*\}.
 \end{equation}
 
 This proposition appeared first in \cite{ABL}.
 
 \begin{pr}\label{pr1}
  Let  $(\G,[\;,\;],\prs)$ be a
 pseudo-Riemannian flat Lie algebra. Then the modular vector satisfies $\mathbf{h}\in[\G,\G]\cap[\G,\G]^\perp$ and 
 $\mathrm{R}_{\mathbf{h}}$ is symmetric with respect to $\prs$.
 In
 particular, if $\G$ is nonunimodular then $[\G,\G]$ is degenerate and $\langle
 \mathbf{h},\mathbf{h}\rangle=0$. 
 
 \end{pr}
 
 {\bf Proof.}  For any $u\in[\G,\G]^\perp$ and any $v\in\G$, since $\mathrm{R}_u$ is symmetric, we have
 $ \langle u.u,v\rangle=\langle u,v.u\rangle=0, $ and hence $u.u=0$.
  So, by virtue of \eqref{left}, we get
 $[\mathrm{R}_u,\mathrm{L}_u]=\mathrm{R}_u^{2}.$
 One can deduce by induction that, for any $k\in\N^*$,
 $[\mathrm{R}_u^k,\mathrm{L}_u]=k\mathrm{R}_u^{k+1},$ and hence
 $\tr(\mathrm{R}_u^{k})=0$
 for any $k\geq2$ which implies that $\mathrm{R}_u$ is nilpotent.
 Since, for any $u,v\in\G$, $\tr(\ad_{[u,v]})=0$, we deduce that
 $\mathbf{h}\in[\G,\G]^\perp$.
 Now, for any $u\in[\G,\G]^\perp$, $\mathrm{R}_u$ is nilpotent and hence
 $$\tr(\ad_u)=-\tr(\mathrm{R}_u)=\langle \mathbf{h},u\rangle=0,$$which implies
 $\mathbf{h}\in[\G,\G]$.\hfill $\square$\\
\subsection{Proof of Theorem \ref{main}} 
  We begin by proving the theorem's first assertion.\\
 Let $(\G,\prs)$ be a nonunimodular Lorentzian flat Lie algebra. According to \cite{della} Corollary 3.6, $\G$ must be solvable.  The left multiplication $\mathrm{L}:\G\too\mathrm{so}(\G)$ is a representation and hence
 \[ \G=\mathfrak{h}\oplus\mathfrak{k}, \]
 where $\mathfrak{h}$ is $L$-invariant Lorentzian nondegenerate indecomposable and 
 $\mathfrak{k}$ is $L$-invariant Euclidean nondegenerate. It is obvious that ${\mathfrak{k}}$ is a Riemannian flat Lie algebra and hence it is unimodular. Moreover,
   it is easy to see that  $\mathbf{h}\in \mathfrak{h}$ and it coincides with the modular vector of $\mathfrak{h}$. Let us show that $\mathrm{L}_{\mathbf{h}}(\mathfrak{k})=0$. Indeed, according to Proposition
\ref{pr2} and Corollary \ref{co1}, \begin{equation}\label{k} 
\mathfrak{k}=\bigoplus_{i=1}^p\mathfrak{k}_i\oplus \mathfrak{k}_0,\end{equation} and for any
$i=1,\ldots,p$, there exists $\mu_i>0$ with $d\mu_i=0$ and an orthonormal basis $(s_i,t_i)$ of $\mathfrak{k}_i$ such that, for any $u\in\G$,
\begin{equation}\label{st} u.s_i=\mu_i(u)t_i\esp u.t_i=-\mu_i(u)s_i. \end{equation}
Or, by virtue of Proposition \ref{pr1}, $\mathbf{h}\in[\G,\G]$ so $\mu_i(\mathbf{h})=0$ and hence $\mathbf{h}.s_i=\mathbf{h}.t_i=0$.\\
Now, according to Proposition \ref{pr3}, we have three situations.\\
   {\bf First situation.}
 In this case $\mathfrak{h}=\R e$ with $\langle e,e\rangle<0$ and for any $u\in\G$ $u.e=0$. Then $\mathfrak{h}$ is unimodular which is impossible.\\
  {\bf Second situation.}
 There exists $\la>0$ with $d\la=0$,  a basis $(e,\bar{e})$ of $\mathfrak{h}$ such that $\langle e,e\rangle=\langle \bar{e},\bar{e}\rangle=0$, $\langle e,\bar{e}\rangle=1$ and, for any $u\in\G$, 
 $$ u.e=\la(u){e}\esp u.\bar{e}=-\la(u)\bar{e}. $$
  We have 
 $ [e,\bar{e}]=-\la(e)\bar{e}-\la(\bar{e}) {e}.$
 Since $\la([e,\bar{e}])=0$, we get $\la(e)\la(\bar{e})=0.$ We can suppose without loss of generality that $\la(e)=0$. So
 \[ e.e=e.\bar{e}=0,\; \bar{e}.e=\la(\bar{e})e\esp \bar{e}.\bar{e}=-\la(\bar{e})\bar{e}. \]
 So $\mathbf{h}=\la(\bar{e}) e$ and hence $\la(\bar{e})\not=0$.  Thus $\mathrm{L}_{\mathbf{h}}=0$. This shows that $H$ is a two-sided ideal. Moreover, for any $u\in\G$, from
 \[ (\bar{e}.u).\bar{e}-\bar{e}.(u.\bar{e})=(u.\bar{e}).\bar{e}-u.(\bar{e}.\bar{e}), \]
 we get
 \[ \la(\bar{e}.u)=-\la(u)\la(\bar{e}). \]
 If $u\in\mathfrak{k}_0$, we get $\la(u)=0$. For $i=1,\ldots,p$, by replacing $u$ by $s_i$ and $t_i$ respectively, we get
 \[ \mu_i(\bar{e})\la(t_i)+\la(s_i)\la(\bar{e})=0\esp
 \la(t_i)\la(\bar{e})-\mu_i(\bar{e})\la(s_i)=0.  \]
 Since $\la(\bar{e})\not=0$ we get $\la(t_i)=\la(s_i)=0$. Now, $H^\perp=\R e\oplus \mathfrak{k}$ and hence $H^\perp.H=0$. This shows that
  $H^\perp$ is also a two-sided ideal.\\
   {\bf Third situation.} In this case $\dim\mathfrak{h}\geq3$ and 
 there exists $\la\in\G^*$ such that $d\la=0$ and $\mathfrak{h}_\la$ is a totally isotropic one dimensional vector space. Moreover, for any $\mu\not=\la$, $\mathfrak{h}_\mu=\left\{0\right\}$. Choose a generator $e$ of $\mathfrak{h}_\la$. 
 Let  $\pi:\mathfrak{h}_\la^\perp\too \wi{\mathfrak{h}}$ the canonical projection where 
 $\wi{\mathfrak{h}}$ is the quotient of $\mathfrak{h}_\la^\perp$ by $\mathfrak{h}_\la$. The representation $\mathrm{L}$ induces a representation $\wi{\mathrm{L}}$ of $\G$ on $\wi{\mathfrak{h}}$. So, according to Corollary \ref{co1},
 \[ \wi{\mathfrak{h}}=\bigoplus_{i=1}^q\wi{\mathfrak{h}}_i\oplus \wi{\mathfrak{h}}_0. \]
 Put $E_i=\pi^{-1}(\wi{\mathfrak{h}}_i)$ for $i=0,\ldots,q$. For any $x\in E_0$ there exists $a_x\in\G^*$ satisfying \eqref{ax} such that, for any $u\in\G$,
 \[ u.x=a_x(u)e. \]
   We have clearly   
 \begin{equation}\label{h} \langle \mathbf{h},e\rangle=-\la(e). \end{equation}
 We distinguish two cases.
 
 {\bf First case: $\la(e)\not=0$.} We will show that this case is impossible.
 Put $\bar{e}=-\la(e)^{-1}\mathbf{h}$. For any $i=1,\ldots,q$, there exists $\la_i>0$,  $b_i,c_i\in\G^*$ satisfying \eqref{bci} and a basis $(e,e_i,f_i)$ of $E_i$ such that $(e_i,f_i)$ is orthonormal and for any $u\in\G$,
 \[ u.e_i=b_i(u)e+\la_i(u)f_i\esp u.f_i=c_i(u)e-\la_i(u)e_i. \]
 Moreover, according to \eqref{changebc}, we can choose $b_i$ and $c_i$ such that $b_i(e)=c_i(e)=0$.\\ Fix $i=1,\ldots,q$. By using the flatness of the metric, we will show that $b_i=c_i=0$. Indeed,
  for any $u\in\G$.
 \begin{eqnarray*}
  (e.u).e-e.(u.e)&=&(u.e).e-u.(e.e).
  \end{eqnarray*}So
  \begin{equation}\label{levi1}
  \la(e.u)=\la(u)\la(e).\end{equation}
 Also, for any $x\in E_0$,
  \begin{eqnarray*}
  (x.u).e-x.(u.e)&=&(u.x).e-u.(x.e).
  \end{eqnarray*}So
  \begin{equation}\label{levi2} \la(x.u)=a_x(u)\la(e). \end{equation}
 Finally, 
 \begin{eqnarray*}
 (u.e).e_i-u.(e.e_i)&=&(e.u).e_i-e.(u.e_i)\\
 (u.e).f_i-u.(e.f_i)&=&(e.u).f_i-e.(u.f_i),
 \end{eqnarray*}then
 \begin{equation}\label{eqe}
 \left\{\begin{array}{lll} 
 \la_i(e.u)&=&\la(u)\la_i(e),\\
 b_i(e.u)&=&\la(e)b_i(u) -\la_i(e)c_i(u),\\
 c_i(e.u)&=&\la_i(e)b_i(u)+\la(e)c_i(u).\end{array}
     \right.
 \end{equation}
 By taking $u\in E_0\oplus \mathfrak{k}_0$ in \eqref{eqe},
  we get
   \[ b_i(u)\la(e)-\la_i(e)c_i(u)=0\esp  \la_i(e)b_i(u)+c_i(u)\la(e)=0. \]
 So $b_i(u)=c_i(u)=0$.

 Let $j=1,\ldots,q$. By taking $u=e_j$ or $u=f_j$ in \eqref{eqe}, we get
  \[ \left( \begin{array}{cccc}
 \la(e)&-\la_j(e)&-\la_i(e)&0\\
 \la_j(e)&\la(e)&0&-\la_i(e)\\
 \la_i(e)&0&\la(e)&-\la_j(e)\\
 0&\la_i(e)&\la_j(e)&\la(e)
 \end{array} \right) 
 \left( \begin{array}{c}b_i(e_j)\\b_i(f_j)\\c_i(e_j)\\c_i(f_j)
 \end{array} \right)=0.\]
 Since $\la(e)\not=0$ we get $b_i(e_j)=b_i(f_j)=c_i(e_j)=c_i(f_j)=0$. In the same way, we can show that $b_i(s_j)=b_i(t_j)=c_i(s_j)=c_i(t_j)=0$. Recall that $(s_i,t_i)$ are defined in \eqref{st}.\\
 On the other hand,
 $ e.\bar{e}=-\la(e)\bar{e}+x_0,$ where $x_0\in\mathfrak{h}_0$, so
 \[ b_i(\bar{e})\la(e)-\la_i(e)c_i(\bar{e})=-\la(e)b_i(\bar{e})\esp  \la_i(e)b_i(\bar{e})+c_i(\bar{e})\la(e)=-\la(e)c_i(\bar{e}). \]
 So $b_i(\bar{e})=c_i(\bar{e})=0$.
  Finally, for any $i=1,\ldots,q$, $b_i=c_i=0$ and hence $\mathrm{span}\{e_i,f_i\}$ is a nondegenerate invariant subspace. Since $\mathfrak{h}$ is indecomposable then   $\mathfrak{h}=E_0\oplus\R \mathbf{h}$ and $E_0=\R e\oplus F_0$ such that $\la(F_0)=0$.
  For any $x\in F_0$, by taking $u=e$ in \eqref{levi2} we get $a_x(e)=0$. By applying $da_y$ to $(e,x)$ and
 by using \eqref{ax} we get, for any $x,y\in F_0$,
  $ 0=\la(e)a_y(x)$ and hence 
  $a_y(x)=0$. By taking $u\in \mathfrak{k}_0$ in \eqref{levi2} we get $a_x(u)=0$. For any $i=1,\ldots,p$, by  taking $u=s_i$ or $u=t_i$ we get $a_x(s_i)=a_x(t_i)=0$. Now, $\mathbf{h}\in[\G,\G]^\perp$ and hence $\mathrm{R}_{\mathbf{h}}$ is symmetric, so
 \[ 0=\langle x.\mathbf{h},\mathbf{h}\rangle=\langle x,\mathbf{h}.\mathbf{h}\rangle=-\langle \mathbf{h}.x,\mathbf{h}\rangle
 \stackrel{\eqref{h}}=a_x(\mathbf{h})\la(e)=0. \]So $a_x(\mathbf{h})=0$. Thus $a_x=0$ and  since $\mathfrak{h}$ is indecomposable then $\mathfrak{h}=\mathrm{span}\left\{ e,\bar{e}\right\}$ which is impossible.

{\bf Second case: $\la(e)=0$.} In this case, by virtue of \eqref{h} and since $\langle \mathbf{h},\mathbf{h}\rangle=0$, $\mathbf{h}=\al e$ and hence $e\in[\G,\G]$ and $\mathrm{R}_e$ is symmetric. So for any $u,v\in\G$,
 \[ \la(u)\langle e,v\rangle=\la(v)\langle e,u\rangle. \]
 So if $u\in\mathfrak{h}_\la^\perp\oplus\mathfrak{k}$ then $\la(u)=0$ and hence $\ker\la=
 \mathfrak{h}_\la^\perp\oplus\mathfrak{k}$. Since $e\in[\G,\G]$ then for any $i=1,\ldots,q$, $\la_i(e)=0$. Fix $i=1,\ldots,q$ and choose a basis $(e,e_i,f_i)$ of $E_i$ such that for any
 $u\in\G$,
 \[ u.e=\la(u)e,\; u.e_i=b_i(u)e+\la_i(u)f_i\esp u.f_i=c_i(u)e-\la_i(u)e_i. \] 
 So \begin{eqnarray*}
 \;[e,e_i]&=&b_i(e)e,\\
 \;[e,f_i]&=&c_i(e)e,\\
 \;[e_i,f_i]&=&(c_i(e_i)-b_i(f_i))e-\la_i(e_i)e_i-\la_i(f_i)f_i.\end{eqnarray*}
 Since $d\la_i=0$, by applying $\la_i$ to $[e_i,f_i]$, we get $\la_i(e_i)=\la_i(f_i)=0$. 
 By applying $db_i$ and $dc_i$ to $[e,e_i]$ and $[e,f_i]$ and by using \eqref{bci}, we get
 \[ b_i(e)=c_i(e)=0. \]
On the other hand, for any $x\in E_0$, we have $[e,x]=a_x(e)e$. So by applying $da_x$ and by using \eqref{ax} we get $a_x(e)=0$. So $e.\mathfrak{h}_\la^\perp=0$. Now, for any $u\in\mathfrak{h}\setminus\mathfrak{h}_\la^\perp$, $\langle e.u,u\rangle=0$ and, for any $v\in\mathfrak{h}_0$, $\langle e.u,v\rangle=-\langle e.v,u\rangle=0$ and hence $e.u=0$.  Finally, $\mathrm{L}_{\mathbf{h}}=0$. This shows that $H$ is a two-sided ideal and because $H^\perp.H=0$ then $H^\perp$ is also a two-sided ideal. This achieves the proof of the theorem's first assertion.\\
 Since  both $H$ and $H^\perp$ are two-sided ideals, according to Proposition 3.1 of \cite{Aub-Med},  
$\G$  is  a double extension of a Riemannian flat Lie algebra $B$ according to
$(\xi,D,\mu,b_0)$. The Lie bracket is given by
$$[\bar{e},e]=\mu e,\; [\bar{e},a]=D(a)-\langle
b_0,a\rangle_0e\esp[a,b]=[a,b]_0+\langle(\xi-\xi^*)(a),b\rangle_0e.$$
It is easy to see that $\mathbf{h}=(\mu+\tr(D))e$. So $\G$ is nonunimodular iff $\tr(D)\not=-\mu$.
\hfill$\square$

\section{Nonunimodular Lorentzian flat Lie algebras up to dimension 4}\label{section5}
According to theorem \ref{main}, one can determine entirely all nonunimodular Lorentzian flat Lie algebras if one can find all admissible $(\xi,D,\mu,b_0)$ on Riemannian flat Lie algebras with $\tr(D)\not=-\mu$. In this section, we give all families of non unimodular Lorentzian flat Lie algebras up to dimension 4 which are not isomorphic.\\
Not first that in the 2-dimensional case, there is one nonunimodular Lorentzian flat Lie algebra $(\G_2,\prs)$ such that $\G_2=span\{e,\bar{e}\}$ and $[\bar{e},e]=\mu e$ with $\mu\neq0$ and $\prs_{\{e,\bar{e}\}}=\left(\begin{array}{cc}0&1\\1&0\end{array}\right)$.
\begin{pr}
Let $(\G,\prs)$ be a nonunimodular Lorentzian flat Lie algebra of dimension 3. Then $(\G,\prs)$ is isomorphic to one of the following:
\begin{enumerate}
\item[$(i)$] $\G_3=span\{e,\bar{e},e_1\}$ where the only non vanishing brackets are 
$$[\bar{e},e]=\mu e\ ,\ [\bar{e},e_1]=\alpha e\ ,\ \mbox{with }\mu\neq0\mbox{ and }\alpha\in\mathbb{R}.$$ 
\item[$(ii)$] $\G'_3=span\{e,\bar{e},e_1\}$ where the only non vanishing brackets are 
$$[\bar{e},e]=\mu e\ ,\ [\bar{e},e_1]=\mu e_1+\alpha e\ ,\ \mbox{with }\mu\neq0\mbox{ and }\alpha\in\mathbb{R}.$$
\end{enumerate}
In both cases, the metric is given by
$$\prs_{\{e,\bar{e},e_1\}}=\left(\begin{array}{ccc}0&1&0\\1&0&0\\0&0&1\end{array}\right)$$
\end{pr} 

\begin{proof}
It is easy to show that $(\xi,D,\mu,b_0)$ is admissible in a one dimensional Riemannian Lie algebra $B=\mathbb{R}e_1$ if and only if $D=\xi=0$ or $D=\xi=\mu Id_B$. Put $b_0=\alpha e_1$, and by using \eqref{bracket} we show that $\G$ is isomorphic to $\G_3$ or $\G'_3$. Since $\dim [\G_3,\G_3]\neq\dim [\G'_3,\G'_3]$, then $\G_3$ and $\G'_3$ are not isomorphic.
\end{proof}

\begin{pr}
Let $(\G,\prs)$ be a nonunimodular Lorentzian flat Lie algebra of dimension 4. Then $(\G,\prs)$ is isomorphic to one of the following:
\begin{enumerate}
\item[$(i)$] $\G_4=span\{e,\bar{e},e_1,e_2\}$ where the only non vanishing brackets are 
$$[\bar{e},e]=\mu e\ ,\ [\bar{e},e_1]=\mu e_1+\lambda e_2+\alpha e\ ,\ [\bar{e},e_2]=\beta e\ ,\ [e_1,e_2]=\lambda e$$ with $\mu\neq0$ and $\alpha,\beta,\lambda\in\mathbb{R}.$ 
\item[$(ii)$] $\G'_4=span\{e,\bar{e},e_1,e_2\}$ where the only non vanishing brackets are 
$$[\bar{e},e]=\mu e\ ,\ [\bar{e},e_1]=\e\mu e_1+\gamma e_2+\alpha e\ ,\ [\bar{e},e_2]=-\gamma e_1+\e\mu e_2+\beta e,$$ with $\e=0$ or $1$, $\mu\neq0$ and $\alpha,\beta,\gamma\in\mathbb{R}.$
\end{enumerate}
In both cases, the metric is given by
$$\prs_{\{e,\bar{e},e_1,e_2\}}=\left(\begin{array}{cccc}0&1&0&0\\1&0&0&0\\0&0&1&0\\0&0&0&1\end{array}\right)$$
\end{pr} 
 \begin{proof}
 Any Riemannian flat Lie algebra $(B,\prs)$ of dimension 2 must be abelian, then $(D,\xi,\mu,b_0)$ is admissible if and only if $A=D-\xi$ is skew-symmetric and 
 \begin{equation}\label{axi}
 [A,\xi]=\xi^2-\mu\xi.
 \end{equation}
 Put $$A=\left(\begin{array}{cc}0&\gamma\\-\gamma&0\end{array}\right)\esp \xi=\left(\begin{array}{cc}\lambda_1&\lambda_2\\\lambda_3&\lambda_4\end{array}\right),$$
 then from \eqref{axi} we deduce that $(\lambda_2-\lambda_3)(\lambda_1+\lambda_4-\mu)=0.$
 \begin{itemize}
 \item If $\lambda_1+\lambda_4=\mu$ then $A=0$. Indeed, if $\gamma\neq0$ then by using \eqref{axi}, we get $\lambda_1=\lambda_4=\mu=0$, but $\tr(D)=\tr(\xi)=\lambda_1+\lambda_4\neq-\mu.$\\
 Then \eqref{axi} is equivalent to $A=0$ and $\det \xi=0$. Thus $(A=\xi=0)$ or $(A=0\mbox{ and } \xi=\left(\begin{array}{cc}\mu&0\\\lambda&0\end{array}\right)\mbox{ where } \mu\neq0 \mbox{ and } \lambda\in\mathbb{R})$.\\
 Put $b_0=\alpha e_1+\beta e_2$, then by using \eqref{bracket}, we see that the second solution corresponds to $\G_4$, and the first solution gives a particular case of $\G'_4$.
 \item If $\lambda_1+\lambda_4\neq\mu$ then $\lambda_2=\lambda_3$ which implies that $\xi$ is symmetric. Then there exists an orthonormal basis of $B$ such that 
 $$A=\left(\begin{array}{cc}0&\gamma\\-\gamma&0\end{array}\right)\esp \xi=\left(\begin{array}{cc}\lambda_1&0\\0&\lambda_2\end{array}\right),$$
 then the solutions of \eqref{axi} are:
 \begin{enumerate}
 \item $A$ is skew-symmetric and $\xi=\mu Id_B$ or $\xi=0$ which corresponds to $\G'_4$.
 \item $A$ is skew-symmetric and $\xi=\left(\begin{array}{cc}\mu&0\\0&0\end{array}\right)$ which corresponds to a particular case of $\G_4$.
 \end{enumerate}
 \end{itemize} 
 Finally, $\G_4$ and $\G'_4$ are not isomorphic because $\dim [\G_4,\G_4]=2$ and $\dim [\G'_4,\G'_4]=1$ or $3$.
 \end{proof}
 
\bibliographystyle{elsarticle-num}

\end{document}